\documentclass[11pt]{amsart}

\usepackage{amsmath, amssymb, amsthm}
\usepackage[foot]{amsaddr}
\usepackage{epsfig}
\usepackage{geometry}
\usepackage{amsfonts}
\usepackage{enumitem}
\usepackage{bbm}
\usepackage[font=small,labelfont=bf, width=.7\textwidth]{caption}
\newtheorem{theorem}{Theorem}
\newtheorem{lemma}[theorem]{Lemma}
\newtheorem{corollary}[theorem]{Corollary}
\newtheorem{proposition}[theorem]{Proposition}

\newtheorem{remark}[theorem]{Remark}

\newtheorem{conjecture}{Conjecture}

\newtheorem{quest}{Question}

\DeclareMathOperator{\im}{im}

\begin{document}

\title{Clique immersion in graph products}
\author[Collins]{Karen L. Collins}
\address[Karen L. Collins]{Department of Mathematics and Computer Science, Wesleyan University, Middletown, CT, USA 06459} \email[Karen L. Collins]{kcollins@wesleyan.edu}

\author[Heenehan]{Megan E. Heenehan}
\address[Megan E. Heenehan]{Department of Mathematical Sciences, Eastern Connecticut State University, Willimantic, CT, USA 06226} \email[Megan Heenehan]{heenehanm@easternct.edu}

\author[McDonald]{Jessica McDonald} \address[Jessica McDonald]{Department of Mathematics and Statistics, Auburn University, Auburn, AL, USA 36849}\email[Jessica McDonald]{mcdonald@auburn.edu}

\thanks{The third author is supported in part by NSF grant DMS-1600551}

\date{}

\maketitle

\begin{abstract}  Let $G,H$ be graphs and $G*H$ represent a particular graph product of $G$ and $H$. We define $\im(G)$ to be the largest $t$ such that $G$ has a $K_t$-immersion and ask: given $\im(G)=t$ and $\im(H)=r$, how large is $\im(G*H)$? Best possible lower bounds are provided when $*$ is the Cartesian or lexicographic product, and a conjecture is offered for each of the direct and strong products, along with some partial results.
\end{abstract}

\section{Introduction}

In this paper every graph is assumed to be simple.

Formally, a pair of adjacent edges $uv$ and $vw$ in a graph are \emph{split off} (or \emph{lifted}) from their common vertex $v$ by deleting the edges $uv$ and $vw$, and adding the edge $uw$. Given graphs $G, G'$, we say that $G$ has a \emph{$G'$-immersion} if a graph isomorphic to $G'$ can be obtained from a subgraph of $G$ by splitting off pairs of edges, and removing isolated vertices. We define the \emph{immersion number} of a graph $G$, denoted $\im(G)$, to be the largest value $t$ for which $G$ has an $K_t$-immersion. We call the $t$ vertices corresponding to those in the $K_t$-immersion the \emph{terminals} of the immersion.

Immersions have enjoyed increased interest in the last number of years (see eg. \cite{CDKS, CH, DDFMMS, DW, DY, KK, V17, Wa, Wo}). A major factor in this was Robertson and Seymour's \cite{RS23} proof that graphs are well-quasi-ordered by immersion, published as part of their celebrated graph minors project (where they show that graphs are well-quasi-ordered by minors). Although graph minors and graph immersions are incomparable, it is interesting to ask the same questions about both. In the realm of minors, motivated by Hadwiger's conjecture~\cite{hadwiger}, authors have asked: what is the largest complete graph that is a minor of a given graph? In this paper we ask: what is the largest complete graph that is immersed in a given graph? Similar questions were also asked about subdivisions after Haj\'os~\cite{Haj} conjectured that a graph with chromatic number $n$ would have a subdivision of a $K_n$. However, this conjecture is false for $n\geq 7$ \cite{C}. Since every subdivision is an immersion, but not every immersion is a subdivision, we examine whether or not the counterexamples provided by Catlin in \cite{C} have immersion numbers of interest.

In this paper, we are interested in the immersion number of graph products. In particular, for graphs $G$ and $H$, we consider the four standard graph products: the \emph{lexicographic product} $G\circ H$, the \emph{Cartesian product} $G\Box H$, the \emph{direct product} $G\times H$, and the \emph{strong product} $G \boxtimes H$.  The central question of this paper is the following.

 \begin{quest}\label{ques}
 Let $G$ and $H$ be graphs with $im(G)=t$ and $im(H)=r$. For each of the four standard graph products, $G*H$, is $\im(G*H)\geq\im(K_t*K_r)$?
 \end{quest}

In this paper we determine that the answer to Question~\ref{ques} is yes for the lexicographic and Cartesian products. In addition we provide partial results for the direct product and conjecture that the answer is yes for the direct and strong products. In determining the immersion number for $K_t*K_r$, for any product *, we choose as our terminals a vertex and all of its neighbors. In trying to affirmatively answer Question~\ref{ques}, we use a similar strategy for choosing terminals.

We will now summarize the results in each section of the paper. In Section~\ref{prelims}, we describe necessary background, including an alternate definition of graph immersion that we use throughout the rest of the paper, and explain our strategy for choosing terminals in a graph product. In Section~\ref{Lexicographic}, we discuss the lexicographic product and affirmatively resolve Question~\ref{ques} for the lexicographic product of two or more graphs (Theorem~\ref{lex}). There is an appealing immersion-analog of the Haj\'os Conjecture \cite{Haj} by Abu-Khzam and Langston \cite{AKL}, namely, that $\chi(G)\geq t$  implies $\im(G)\geq t$. While the  Haj\'os Conjecture  was disproved by Catlin \cite{C} using lexicographic products as counterexamples, in Section~\ref{lex}, we show
that the lexicographic product does not yield smallest counterexamples to the Abu-Khzam and Langston conjecture.

In Section~\ref{Cartesian}, we discuss the Cartesian product. In Section~\ref{CartBound}, we  affirmatively resolve Question~\ref{ques} for the Cartesian product of two graphs (Theorem~\ref{box}) and provide a contrasting example of an immersion number greater than that given in the theorem. In Section~\ref{CartPowers}, we extend our results to products with more than two factors. In particular, we show the immersion number of the $d$ dimensional hypercube, $Q_d$, is $d+1$ and the immersion number of the Hamming graph, $K_n^d$, is $d(n-1)+1$. For the Cartesian product of a path on $n$ vertices with itself $d$ times, denoted $P_n^d$, we show $\im(P_n^d)=2d+1$. We compare the results for hypercubes, Hamming graphs, and $P_n^d$ to results by Kotlov~\cite{Ko01}  and  Chandran and Sivadasan~\cite{ChSi07} for graph minors. In addition, we show we can do better than the bound of Theorem~\ref{box}  by proving $\im(G\Box P_n)\geq t+2$ when $\im(G) =t$.

In Section~\ref{Direct}, we conjecture that the answer to Question~\ref{ques} is yes for the direct product of two graphs and provide partial results towards the proof of this conjecture (Conjecture~\ref{direct}). In Section~\ref{DirectComplete}, we find the immersion number of $K_t\times K_r$ (Theorem~\ref{KtTimesKr}) and thus that the conjecture holds when $G$ and $H$ contain $K_t$ and $K_r$ as subgraphs, respectively (Corollary~\ref{CompleteSubgraphTimes}). In addition, we prove that the conjecture holds when $r\geq 3$ and  $K_r$ is a subgraph of $H$ (Theorem~\ref{GtimesKr}). In Section~\ref{PegParity}, we extend these results to $G$ and $H$ having immersions in which all of the paths have the same parity (Theorem~\ref{DirectParity} and Theorem~\ref{GTimesParity}). In Section~\ref{DirectEx}, we provide some examples. In Section~\ref{limitations}, we discuss the cases that remain to prove the conjecture and the limitations of our current proof techniques.

Finally in Section~\ref{final}, we end with some concluding remarks and a conjecture about the strong product.

\section{Preliminaries}\label{prelims}
In this paper all graphs are finite and simple. For graph products we follow the notation of \cite{HIK} and \cite{ImKl00}.

One definition of immersion was provided in the Introduction; an alternative definition for graph immersion is as follows. Given graphs $G$ and $G'$, $G$ has a $G'$-immersion if there is an injective function $\phi: V(G')\to V(G)$ such that for each edge $uv\in E(G')$, there is a path in $G$ joining vertices $\phi(u)$ and $\phi(v)$, and these paths are edge-disjoint for all $uv\in E(G')$. We denote the path from $u$ to $v$ in $G$ by $P_{u,v}$. In this context we call the vertices of $\{\phi(v): v\in V(G)\}$ the \emph{terminals} (or \emph{corners}) of the $G'$-immersion, and we call internal vertices of the paths $\{P_{u,v}: uv\in E(G)\}$ the \emph{pegs} of the $G'$-immersion. In this paper we will often refer to the terminals and pegs of an immersion.

For the reader's convenience, we begin by providing a definition of each of the four standard graph products. Each graph product is defined to have vertex set $V(G)\times V(H)$. Two vertices $(g, h)$ and $(g', h')$ are defined to be adjacent if: $gg'\in E(G)$ or $g=g'$ and $hh' \in E(H)$ (lexicographic product); $g=g'$ and $hh' \in E(H)$, or $gg' \in E(G)$ and $h=h'$ (Cartesian product); $gg' \in E(G)$ and $hh' \in E(H)$ (direct product). The edge set of the strong product is defined to be $E(G \Box H) \cup E(G \times H)$.

In order to contain a $K_n$-immersion, a graph must not only have at least $n$ vertices, but it must have  at least $n$ vertices whose degree is at least $n-1$. In particular, this gives the following observation.

\begin{remark}\label{Delta} For every graph $G$, $\im(G) \leq \Delta(G)+1$.
\end{remark}

Given Remark \ref{Delta} we make the following proposition for a bound on the immersion number of each product.

\begin{proposition}\label{MaxProducts} Given two graphs $G$ and $H$ where $n$ is the number of vertices in $H$,
\begin{enumerate}
  \item $\im(G\circ H)\leq\Delta(H)+n\Delta(G)+1$,
  \item $\im(G\Box H) \leq \Delta(G) + \Delta(H) +1$,
  \item $\im(G\times H) \leq \Delta(G)\Delta(H) +1, and$
  \item $\im(G\boxtimes H) \leq \Delta(G)\Delta(H) +\Delta(G) + \Delta(H)+1$.
\end{enumerate}
\end{proposition}

\begin{proof} Let $G$ and $H$ be graphs with maximum degrees $\Delta(G)$ and $\Delta(H)$ respectively such that $H$ has $n$ vertices.

\underline{Case 1:} By the definition of the lexicographic product, $\Delta(G\circ H) = \Delta(H) + n\Delta(G)$. Therefore, by Remark \ref{Delta}, $\im(G\circ H) \leq \Delta(H) + n\Delta(G) +1$.

\underline{Case 2:} By the definition of the Cartesian product, $\Delta(G\Box H) = \Delta(G) + \Delta(H)$. Therefore, by Remark \ref{Delta}, $\im(G\Box H) \leq \Delta(G) + \Delta(H) +1$.

\underline{Case 3:} By definition of the direct product $\Delta(G \times H) = \Delta(G)\Delta(H)$. Therefore, by Remark \ref{Delta}, $\im(G\times H) \leq \Delta(G)\Delta(H) +1$.

\underline{Case 4:} By definition of the strong product $\Delta(G \boxtimes H) = \Delta(G)\Delta(H)+\Delta(G)+\Delta(H)$. Therefore, by Remark \ref{Delta}, $\im(G\boxtimes H) \leq \Delta(G)\Delta(H) +\Delta(G)+\Delta(H)+1$.
\end{proof}

In trying to affirmatively answer Question~\ref{ques}, we use the same general strategy for each product. We consider graphs $G$ and $H$ with $\im(G) = t$ and $\im(H)=r$. We look at a $K_t$-immersion in $G$ with terminals $v_1, v_2, \ldots, v_t$ and a $K_r$-immersion in $H$ with terminals $u_1, u_2\ldots u_r$. We then consider $K_t\times K_r$ with the vertices labeled $v_1, v_2, \ldots, v_t$ in $K_t$ and $u_1, u_2\ldots u_r$ in $K_r$. As terminals for our immersion in $G*H$ we take a vertex from $K_t\times K_r$ (usually $(v_1, u_1)$) and all of its neighbors -- these are vertices in $G*H$ since each $v_i\in V(G)$ and each $u_j\in V(H)$. We are then able to use the $K_t$-immersion in $G$ and the $K_r$-immersion in $H$ along with the structure of the specific product to find paths in $G*H$ connecting these terminals.

We begin with a discussion of the lexicographic product.

\section{Lexicographic Products}\label{Lexicographic}
The lexicographic product is of particular interest because Catlin \cite{C} disproved the Haj\'os Conjecture \cite{Haj},  that is, if $\chi(G) = n$, then $G$ contains a subdivision of $K_n$, using lexicographic products of odd cycles and complete graphs. Every subdivision is an immersion, but not every immersion is a subdivision. Abu-Khzam and Langston conjectured that if $\chi(G) = n$, then $G$ contains an immersion of $K_n$ \cite{AKL}. Theorem~\ref{lex} implies that a lexicographic product is never a smallest counterexample to the Abu-Khzam and Langston conjecture, since if $G$ and $H$ satisfy the conjecture, then $G\circ H$ also satisfies the conjecture.

In the following theorem we prove a lower bound for the immersion number of the lexicographic product of two graphs.

\begin{theorem}\label{lex}
 Let $G$ and $H$ be graphs with $im(G)=t$ and $im(H)=r$. Then $im(G\circ H) \geq tr$.
\end{theorem}

The following global definition of the lexicographic product will be helpful in our proof of the bound. For graphs $G$ and $H$, $G\circ H$ is obtained from a copy of $G$ by replacing each vertex in $G$ with a copy of $H$, and replacing each edge in $G$ with a complete bipartite graph. Note that the lexicographic product is not commutative, so the roles of $G$ and $H$ in this global definition cannot be reversed.

\begin{proof}(Theorem~\ref{lex})
Fix a $K_t$-immersion in $G$ and a $K_r$-immersion in $H$. Consider $G\circ H$ (with the global description given above). We use the $r$ terminals in the copies of $H$ that correspond to the $t$ terminals in $G$ as the terminal vertices of our $K_{tr}$-immersion.

Within each copy of $H$ there is a $K_r$-immersion, which we use to get the required paths between the $r$ terminals that we have chosen in $H$ (for our $K_{tr}$-immersion). Consider now, in $G\circ H$, two copies of $H$ that correspond to terminals $u$ and $w$ in the $K_t$-immersion in $G$. Let $H_v$ be the copy of $H$ that corresponding to a vertex $v$. Since there is a path between $u $ and $w$ in the $K_t$-immersion in $G$, $H_u$ is connected to $H_w$ by a sequence of copies of $H$, where each consecutive pair yields a $K_{r,r}$ between the two copies. There are two cases: (i) $u$ and $w$ are adjacent and (ii) $u$ and $w$ are not adjacent. In the case that $u$ and $w$ are adjacent, then each vertex in the set of $r$ terminals in $H_u$ is adjacent to each vertex in the set of $r$ terminals in $H_w$. In the case that $u$ and $w$ are not adjacent, let the path between $u$ and $w$ in $G$ be $u, v_1, v_2, \ldots, v_t, w$. It is well-known that $K_{r,r}$ has a proper  edge-coloring with $r$ colors. Choose any such $r$-edge coloring of $K_{r,r}$, and apply this coloring to the edges between $H_{v_i}$ and $H_{v_{i+1}}$ for $1\leq i\leq t-1$ and also to the edges between $H_{v_t}$ and $H_w$. Let the terminals in $H_u$ be $z_1, z_2, \ldots z_r$. Color the edge $e$ between $H_u$ and $H_{ v_1}$ by $e$ is colored $i$ if and only if $e$ is incident to $z_i$. Then the resulting edge-coloring gives an $i$-colored path from $z_i$ to each terminal in $H_w$. These edges provide the required edge-disjoint paths between our two copies of $H$. Therefore $G\circ H$ has a $K_{tr}$-immersion.
\end{proof}

Figure~\ref{K_4_4Lex} shows an example of the edge-coloring process described in the above proof, and the paths that are formed.

 \begin{figure}
\centering
\includegraphics[height=6cm]{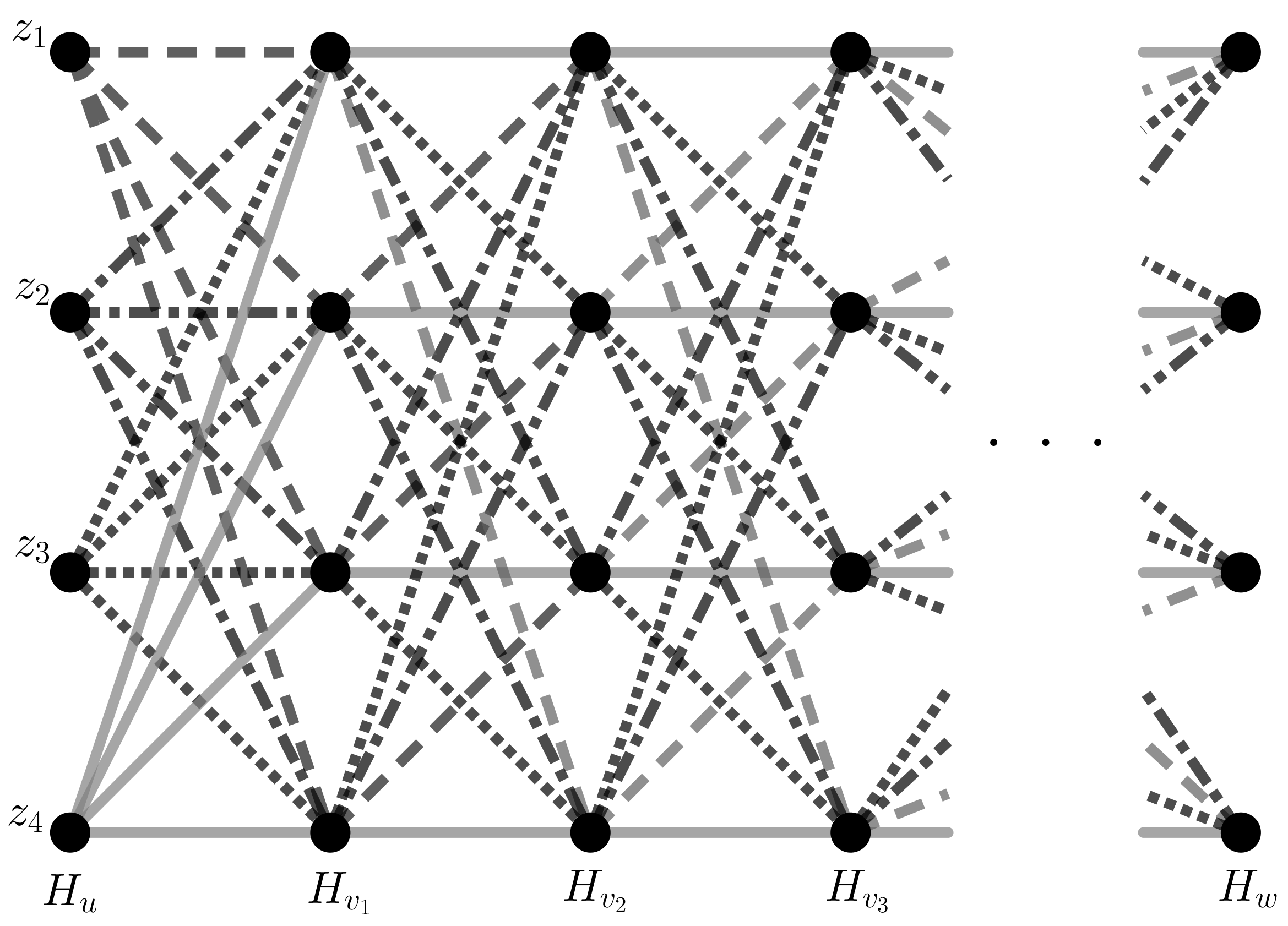}
\caption{Illustration of the edge-coloring described in the proof of Theorem \ref{lex}} when there is a path from $u$ to $w$ in $G$ with $r=4$. The edge colors are followed to form the paths from the terminals $z_1, z_2, z_3,$ and $z_4$ in $H_u$ to the terminals in $H_w$.
\label{K_4_4Lex}
\end{figure}

\begin{corollary} Given graphs $G_1, G_2, \ldots, G_\ell$, $$\im(G_1\circ G_2\circ \ldots \circ G_\ell) \geq \im(G_1)\im(G_2)\ldots\im(G_\ell).$$
\end{corollary}

\begin{proof} Let $G_1, G_2, \ldots ,G_\ell$ be graphs. When $\ell = 2$, Theorem \ref{lex} gives $\im(G_1\circ G_2) \geq \im(G_1)\im(G_2)$. Assume for $k\geq 2$, $\im(G_1\circ G_2\circ \ldots \circ G_k) \geq \im(G_1)\im(G_2)\ldots\im(G_k)$. When $n = k+1$, $G_1\circ G_2\circ \ldots \circ G_k\circ G_{k+1} = (G_1\circ G_2\circ \ldots \circ G_k)\circ G_{k+1}$. Thus, by induction, $\im(G_1\circ G_2\circ \ldots \circ G_k\circ G_{k+1}) \geq\im(G_1)\im(G_2)\ldots\im(G_{k+1})$.
\end{proof}

Theorem \ref{lex} combined with Proposition \ref{MaxProducts} imply the following corollaries.

\begin{corollary} $\im(K_t\circ K_r)=tr$.
\end{corollary}

\begin{proof} By Theorem \ref{lex} $\im(K_t\circ K_r)\geq tr$ and by Proposition~\ref{MaxProducts}, $\im(K_t\circ K_r)\leq (r-1) + r(t-1)+1 = tr$. Therefore $\im(K_t\circ K_r)=tr$.
\end{proof}

\begin{corollary}\label{CycleLexComplete} For $n\geq 3$, $\im(C_n\circ K_r)=3r$.
\end{corollary}

\begin{proof} By Theorem \ref{lex} $\im(C_n\circ K_r)\geq 3r$ and by Proposition~\ref{MaxProducts}, $\im(C_n\circ K_r)\leq (r-1) + r(2)+1 = 3r$. Therefore $\im(C_n\circ K_r)=3r$.
\end{proof}

As an example where we can do better than the bound of Theorem \ref{lex}, consider $K_3\circ C_5$.

\begin{proposition} $\im(K_3\circ C_5) = 13$ \end{proposition}

\begin{proof}
By Proposition~\ref{MaxProducts}, $\im(K_3\circ C_5) \leq 2 + 5(2)+1 = 13$. We now describe the immersion. Label the vertices of $K_3$ as $v_1, v_2,$ and $v_3$. Label consecutive vertices in $C_5$ as $u_1, u_2, u_3, u_4,$ and $u_5$ so that $u_1$ and $u_5$ are adjacent. All vertices are used as terminals  except $(v_2, u_1)$ and $(v_3, u_1)$. Terminals in different copies of $C_5$ are connected by an edge. It remains to connect terminals in the same copy of $C_5$ to each other. To complete the immersion we use the following paths.

$(v_1,u_1) - (v_2, u_1) - (v_1, u_4)$

$(v_1,u_2) - (v_2, u_1) - (v_1, u_5)$

$(v_1, u_2) - (v_3, u_1) - (v_1, u_4)$

$(v_1, u_3) - (v_3, u_1) - (v_1, u_5)$

$(v_1, u_1) - (v_3, u_1) - (v_2, u_1) - (v_1, u_3)$

$(v_2, u_2) - (v_3, u_1) - (v_2, u_4)$

$(v_2, u_3) - (v_3, u_1) - (v_2, u_5)$

$(v_3, u_2) - (v_2, u_1) - (v_3, u_4)$

$(v_3, u_3) - (v_2, u_1) - (v_3, u_5)$
\end{proof}

A similar strategy to the above may be used to show $\im(C_7 \circ C_5) = 13$.

Next we explore the Cartesian product.

\section{Cartesian products}\label{Cartesian}

The following global definition of the Cartesian product will be helpful in our proof of Theorem \ref{box}. Given graphs $G$ and $H$, the graph $G\Box H$ can be obtained from a copy of $H$ by replacing each vertex in $H$ with a copy of $G$, and replacing each edge in $H$ with a perfect matching that pairs identical vertices in the copies of $G$. Since the Cartesian product is commutative, we may switch the roles of $G$ and $H$ without changing the results. In Section~\ref{CartBound}, we discuss bounds for $\im(G\Box H)$. In Section~\ref{CartPowers}, we discuss the immersion numbers of several graph powers.

\subsection{Bounds on the immersion number}\label{CartBound}

We begin by affirmatively answering Question~\ref{ques} for the Cartesian product.

\begin{theorem}\label{box} Let $G$ and $H$ be graphs with $\im(G)=t$ and $\im(H)=r$. Then $\im(G\Box H) \geq t+r-1$.
\end{theorem}

\begin{proof}
 Fix a $K_t$-immersion in $G$ and a $K_r$-immersion in $H$.

Consider $G\Box H$ with the global description described above. Suppose the terminals of the $K_t$-immersion in $G$ are $u_1, u_2, \ldots, u_t$, and suppose the terminals of the $K_r$-immersion in $H$ are $v_1, v_2, \ldots, v_r$. Choose a copy of $G$ that corresponds to a terminal $v_k$ of the $K_r$-immersion in $H$. For the terminals in our $K_{t+r-1}$-immersion, we choose the vertices that correspond to the terminals of the $K_t$-immersion in this copy of $G$,
\begin{equation}\label{vkterms}
(u_1,v_k), (u_2,v_k), \ldots, (u_t,v_k),
\end{equation}
 along with the vertices \begin{equation}\label{ulterms} (u_l, v_1), (u_l, v_2), \ldots, (u_l, v_{k-1}), (u_l, v_{k+1}), \ldots, (u_l, v_r)
 \end{equation}
 where $u_l$ is some fixed terminal of the $K_t$-immersion in $G$.
For the paths between the terminals in set (\ref{vkterms}), use the edge-disjoint paths provided by the $K_t$-immersion in $G$ (keeping $v_k$ constant). For the paths between the terminals in set (\ref{ulterms}), use the paths provided by the $K_r$-immersion in $H$ (keeping $u_l$ constant, and joining $v_1, v_2, \ldots, v_{k-1}, v_{k+1}, \ldots, v_r$). It remains to connect $(u_i, v_k)$ to $(u_l, v_j)$ for $1 \leq i \leq t$ and $1 \leq j \leq r$. For this, first use the path from $v_k$ to $v_j$ in the $K_r$-immersion, keeping $u_i$ constant, to get from $(u_i, v_k)$ to $(u_i, v_j)$. Then use the path from $u_i$ to $u_l$ in the $K_t$-immersion, keeping $v_j$ constant, to get from $(u_i, v_j)$ to $(u_l, v_j)$. Note that the first segment of this path is edge-disjoint from the paths we used to connect the vertices in (\ref{ulterms}), and the second segment of this path is the first time we have used edges within the $v_j$ copy of $H$. Hence we have built a $K_{t+r-1}$-immersion in $G\Box H$ and $\im(G\Box H)\geq t+r-1$.
\end{proof}

The following corollary shows that the above bound is tight for the Cartesian product of two complete graphs.

\begin{corollary} $\im(K_t\Box K_r) = t+r-1$
\end{corollary}

\begin{proof} By Theorem~\ref{box}, $\im(K_t\Box K_r) \geq t+r-1$. By Proposition~\ref{MaxProducts}, $\im(K_t\Box K_r) \leq (t-1)+(r-1)+1 = t+r-1$. Therefore $\im(K_t\Box K_r) = t+r-1$.
\end{proof}

As an example where we can do better than the bound of Theorem~\ref{box}, we now prove $\im(G\Box P_n)\geq t+2$ for $n\geq5$, when $G\neq K_t$.
\begin{theorem}\label{BoxPath} Let $G$ be connected with $\im(G)=t$. Then $\im(G\Box P_n)\geq t+2$ for $n\geq5$ and $G\neq K_t$.
\end{theorem}

\begin{proof} Let $G\neq K_t$ be a connected graph with $\im(G)=t$ and let $v_1, v_2, \ldots, v_t$ be the terminals in a $K_t$-immersion in $G$. Since $G\neq K_t$ there is at least one non-terminal vertex in $G$, call it $x$. Let $u_1, u_2, \ldots, u_n$ be the vertices of $P_n$ in order along the path. We choose as our terminals of the $K_{t+2}$-immersion, $(v_1, u_3), (v_2, u_3), \ldots, (v_t, u_3)$, $(v_1, u_2)$, and $(v_1, u_4)$.

For the paths between $(v_i, u_3)$ and $(v_j, u_3)$ we use the paths provided by the $K_t$-immersion in the copy of $G$ corresponding to $u_3$. In order to complete the immersion we need edge-disjoint paths from $(v_1, u_2)$ and $(v_1, u_4)$ to $(v_j, u_3)$ and between $(v_1, u_2)$ and $(v_1, u_4)$. For the paths from $(v_1, u_j)$, for $j=2$ or $4$, to $(v_i, u_3)$ use the edge-disjoint paths from $(v_1, u_j)$ to $(v_i, u_j)$ in the $u_j$ copy of $G$ and the edge $(v_i, u_j)-(v_i, u_3)$.   For the path from $(v_1, u_2)$ to $(v_1, u_4)$ use the edge $(v_1, u_2)-(v_1, w_1)$ followed by any path to $(x, u_1)$ in the $u_1$ copy of $G$. Then the path $(x, u_1)-(x, u_2)-(x, u_3)-(x, u_4)-(x, u_5)$ followed by any path from $(x, u_5)$ to $(v_1, u_5)$ in the $u_5$ copy of $G$. Then use the edge $(v_1, u_5)-(v_1, u_4)$ to complete the path. This completes the immersion of $K_{t+2}$.
 \end{proof}

As an explanation for why $G\neq K_t$ in Theorem~\ref{BoxPath} we now show $\im(K_t\Box P_n)=t+1$. This is also an example of a graph that does not reach the bound of Proposition~\ref{MaxProducts} because $\Delta(K_t\Box P_n)=t+1$.

\begin{theorem}\label{completeBoxpath} Given a complete graph on $t$ vertices, $K_t$, and a path on $n\geq2$ vertices, $P_n$, $\im(K_t\Box P_n)=t+1$.\end{theorem}

To prove Theorem~\ref{completeBoxpath} we use the following lemma which follows from the \emph{Corner Separating Lemma} found in \cite{CH}.

\begin{lemma}\label{edgecut} Let $G$ be a graph, $C$ a cutset of edges in $G$, and $M$ a connected component of $G-C$. If $G$ has an immersion in which $k$ terminals are in $G-M$ and $j$ terminals are in $M$, then $|C|\geq kj$.\end{lemma}

\begin{proof} Let $G$ be a graph, $C$ be a cutset of edges, and $M$ a connected component of $G-C$. Suppose $G$ has an immersion in which $k$ terminals are in $G-M$ and $j$ terminals are in $M$. Each of the terminals in $G-M$ must be connected by a path to each of the terminals in $M$ and these paths must be edge-disjoint. Therefore, there must be $kj$ edge-disjoint paths from the $k$ terminals in $G-M$ to the $j$ terminals in $M$. Since $G-M$ and $M$ are connected by an edge cutset $C$, each of these $kj$ paths must use a unique edge of $C$. Thus, $|C|\geq kj$.
\end{proof}

\begin{proof} (Theorem \ref{completeBoxpath}) Consider the graph $K_t\Box P_n$ with $n\geq 2$. By Theorem~\ref{box}, $\im(K_t\Box P_n)\geq t+2-1=t+1$. Suppose, for a contradiction, that $\im(K_t\Box P_n)=t+2$. By the global description of the Cartesian product, $K_t\Box P_n$ can be thought of as a sequence of copies of $K_t$ laid out like a path and connected by matchings along the path's edges. Therefore between consecutive copies of  $K_t$ there is an edge cutset of size $t$. Since each copy of $K_t$  has only $t$ vertices and we have $t+2$ terminals total, all of the terminals cannot be in a single copy of $K_t$. Starting at one end of the sequence of copies of $K_t$, find the first copy of $K_t$ that contains a terminal. Suppose this copy of $K_t$ contains $a$ terminals where $1\leq a\leq t+1$, then there are $t+2-a$ terminals that are separated from this copy by an edge cut of size $t$. By Lemma \ref{edgecut}, $t \geq a(t+2-a)$. That is, $a^2-a(t+2)+t \geq 0$. When $a=1$, $a^2-a(t+2)+t\leq 0$. Similarly when $a=t+1$, $a^2-a(t+2)+t\leq 0$. The absolute minimum value of this quadratic occurs at $a=\frac{t+2}{2}$, which is between $1$ and $t+1$. Therefore, $a^2-a(t+2)+t\leq 0$ for the range of interest, this contradicts $\im(K_t\Box P_n)=t+2$. Thus, $\im(K_t\Box P_n)=t+1$.
\end{proof}

\subsection{Powers of graphs and immersion number}\label{CartPowers}

Theorem~\ref{box} combined with Proposition~\ref{MaxProducts} imply the following corollaries. Here $G^{\ell}$ is the Cartesian product of $G$ with itself $\ell$ times.

\begin{corollary}\label{G^d} Given graphs $G_1, G_2, \ldots, G_\ell$, $\im(G_1\Box G_2 \Box \ldots \Box G_{\ell})\geq \sum_{i=1}^\ell G_i - (\ell-1)$. Further the bound is tight if $\Delta(G_i) = \im(G_i)-1$ for each $i$.\end{corollary}

\begin{proof} When $\ell=2$, Theorem~\ref{box} gives $\im(G_1\Box G_2) \geq \im(G_1)+\im(G_2)-1$. Assume $\im(G_1\Box G_2 \Box \ldots \Box G_k)\geq \sum_{i=1}^k G_i - (k-1)$. When $\ell=k+1$, $G_1\Box G_2 \Box \ldots \Box G_k\Box G_{k+1}=(G_1\Box G_2 \Box \ldots \Box G_k)\Box G_{k+1}$. Thus, by Theorem~\ref{box}, $\im(G_1\Box G_2 \Box \ldots \Box G_k\Box G_{k+1})\geq \sum_{i=1}^k G_i - (k-1) +\im(G_{k+1}) - 1= \sum_{i=1}^{k+1} G_i - k$. Therefore, $\im(G_1\Box G_2 \Box \ldots \Box G_{\ell})\geq \sum_{i=1}^\ell G_i - (\ell-1)$.

If $\Delta(G_i)=\im(G_i)-1$ for all $i$, then by induction $\Delta(G_1\Box G_2 \Box \ldots \Box G_{\ell})=\sum_i^\ell\Delta(G_i)$ and by Proposition~\ref{MaxProducts} $\im(G_1\Box G_2 \Box \ldots \Box G_{\ell})= \sum_{i=1}^\ell G_i - (\ell-1)$.
\end{proof}

\begin{remark}\label{hypercube} Since the $d$-dimensional hypercube, $Q_d$, is the Cartesian product of $K_2$ with itself $d$ times Corollary~\ref{G^d} gives $\im(Q_d)=d+1$.
\end{remark}

\begin{remark}\label{hamming} Corollary~\ref{G^d} gives us the immersion number of the Hamming graph (the Cartesian product of $K_n$ with itself $d$ times) $\im(K^d_n)=d(n-1)+1$. \end{remark}

In contrast, when $\Delta(G) \neq \im(G)-1$, then the bound is not tight. As an example  we find the immersion number of the Cartesian product of a path, $P_n$, with itself $d$ times. We begin by showing $\im(P_6^2)=5$.

\begin{proposition}\label{BoxPath6} $\im(P_6^2)=5$.
\end{proposition}

\begin{proof}
Label consecutive vertices in the path $v_1, v_2, \ldots, v_6$. We use $(v_3, v_3)$ and its neighbors $(v_2, v_3), (v_3, v_2), (v_3, v_4), (v_4, v_3)$ as terminals in our immersion of $K_5$. Notice that, $(v_3, v_3)$ is connected by an edge to all of the other terminals. To connect the remaining pairs of terminals we use the following paths (see also Figure~\ref{P_6BoxP_6}).

$(v_2, v_3) - (v_2, v_2) - (v_3, v_2)$

$(v_3, v_2) - (v_4, v_2) - (v_4, v_3)$

$(v_4, v_3) - (v_4, v_4) - (v_3, v_4)$

$(v_3, v_4) - (v_2, v_4) - (v_2, v_3)$

$(v_2, v_3) - (v_1, v_3) - (v_1, v_4) - (v_1, v_5) - (v_2, v_5) - (v_3, v_5) - (v_4, v_5) - (v_5, v_5) - $

$\hspace{.2in}(v_5, v_4) - (v_5, v_3) - (v_4, v_3)$

$(v_3, v_2) - (v_3, v_1) - (v_4, v_1) - (v_5, v_1) - (v_6, v_1) - (v_6, v_2) - (v_6, v_3) - (v_6, v_4) - $

$\hspace{.2in}(v_6, v_5) - (v_6, v_6) - (v_5, v_6) -  (v_4, v_6) - (v_3, v_6) - (v_3, v_5) - (v_3, v_4)$

This completes the description of a $K_5$-immersion in $P_6^2$. Using Remark~\ref{Delta} and the fact that $\Delta(P_6^2)=4$ we can conclude $\im(P_6^2)=5$.
\end{proof}

 \begin{figure}[h!]
\centering
\includegraphics[height=5cm]{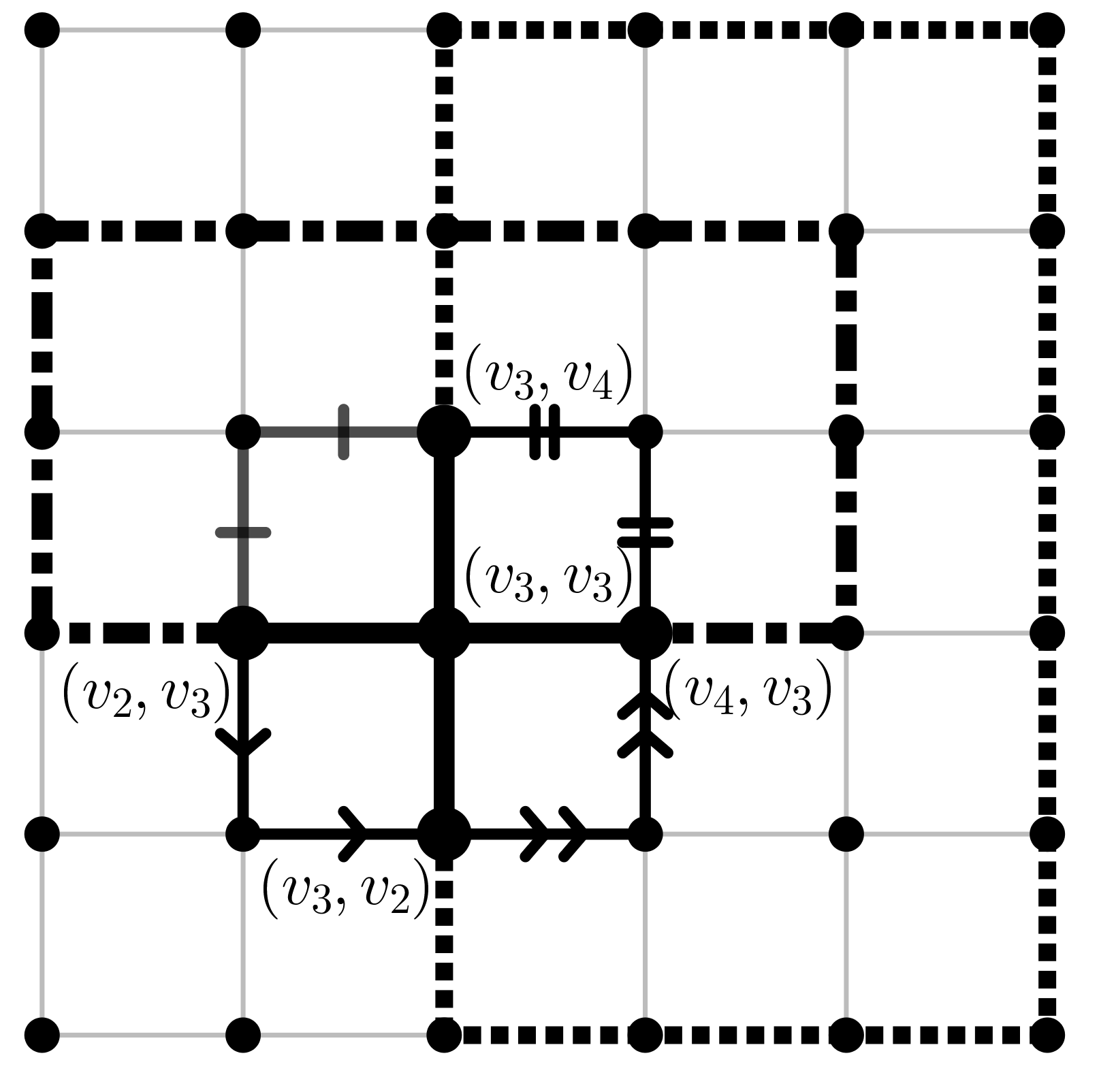}
\caption{An immersion of $K_5$ in $P_6^2$. Terminals are labeled and the edge-disjoint paths are highlighted.}
\label{P_6BoxP_6}
\end{figure}

\begin{corollary}\label{BoxPathd} Let $n\geq 6$ and $d>2$. Then $\im(P_n^d)=2d+1$.
\end{corollary}

\begin{proof}
First we prove that $\im(P_6^d)=2d+1$. Proposition~\ref{BoxPath6} shows $\im(P_6^2)=2(2)+1=5$.
Assume $\im(P_6^k)=2k+1$ for $k\geq 2$. When $d = k+1$, $P_6^{k+1} = P_6^{k}\Box P_6$, thus by Theorem~\ref{BoxPath} and induction, $\im(P_6^{k+1})\geq 2k + 1 + 2 = 2(k+1)+1$.
The maximum degree in $P_6^{k+1}$ is $2(k+1)$.
Therefore, $\im(P_6^{k+1})= 2(k+1)+1$.
Since $P_6^d$ is a subgraph of $P_n^d$ for $n\geq 6$ and $\Delta(P_n^d)=2d$, we have shown $\im(P_n^d)=2d+1$.
\end{proof}

As mentioned in the introduction, Kotolov~\cite{Ko01} and Chandran and Sivadasan~\cite{ChSi07} proved bounds for the Hadwiger number of products of the graphs discussed above. Let $G$ be a graph. The Hadwiger number of $G$, $h(G)$, is the maximum $m$ such that $G$ has a $K_m$-minor. For the $d$-dimensional hypercube, Kotlov proved $h(Q_d)\geq 2^{\frac{d+1}{2}}$ for $d$ odd and $h(Q_d)\geq 3\cdot2^{\frac{d-2}{2}}$ for $d$ even. Chandran and Sivadasan proved $h(Q_d)\leq 2^{\frac{d}{2}}\cdot\sqrt{d}+1$. These results contrast with the result in Remark~\ref{hypercube} where we show $\im(Q_d)=d+1$. Chandran and Sivadasan also showed $n^{\lfloor\frac{d-1}{2}\rfloor}\leq h(K_n^d)\leq h^{\frac{d+1}{2}}\cdot\sqrt{d}+1$ and $n^{\lfloor{\frac{d-1}{2}}\rfloor}\leq h(P_n^d)\leq h^{\frac{d}{2}}\cdot\sqrt{2d}+1$. These results contrast with the results in Remark~\ref{hamming} and Corollary~\ref{BoxPathd}, where we show $\im(K_n^d)=d(n-1)+1$ and $\im(P_n^d)=5$ respectively.

\section{Direct products}\label{Direct}
The bounds for the immersion numbers of the Cartesian and lexicographic products were relatively straightforward to prove. In this section we discuss the direct product. The structure of the direct product is quite different than the previous two products and leads to challenges in proving a bound for the immersion number in the general case.

We begin by conjecturing that the answer to Question~\ref{ques} is yes for the direct product.

\begin{conjecture}\label{direct} Let $G$ and $H$ be graphs where $im(G)=t$ and $im(H)=r$. Then $im(G\times H) \geq (t-1)(r-1)+1$.
\end{conjecture}

The global definition for the direct product of graphs $G$ and $H$ is to form $G\times H$ from a copy of $G$ by replacing each vertex in $G$ with an edgeless copy of $H$, and replacing each edge in $G$ with a set of edges joining vertices $h, h'$ in the two different copies of $H$ when $hh'\in E(H)$. Since the direct product is commutative, we may switch the roles of $G$ and $H$ without changing the results. In this section we present evidence towards the proof of Conjecture~\ref{direct}. In Section~\ref{DirectComplete}, we consider cases where the graphs are complete or have a subgraph of a complete graph of the same size as the immersion number. In Section~\ref{PegParity}, we consider cases involving the parity of the number of pegs on each path in an immersion.

\subsection{Direct products of complete graphs} \label{DirectComplete}
We begin by proving the immersion number for the direct product of two complete graphs.

\begin{theorem}\label{KtTimesKr} $\im(K_t\times K_r) = (t-1)(r-1)+1.$
\end{theorem}

\begin{proof}  Observe that when $t=r=2$, $\im(K_2\times K_2) = 2$ since $K_2\times K_2$ is two disjoint edges.

We now consider the case when at least one of $t$ or $r$ is greater than $2$. By Proposition~\ref{MaxProducts}, $\im(K_t\times K_r)\leq (t-1)(r-1)+1$. To complete the proof we define a $K_{(t-1)(r-1)+1}$-immersion in $K_t\times K_r$.

 Label the vertices of each complete graph $1, 2, \ldots, t$ or $r$. The $(t-1)(r-1)+1$ terminals of our clique immersion will be $(1,1)$ and all its neighbors. Let $N\left[(1,1)\right]=\{(i,j)\;|\; 2\leq i\leq t, 2\leq j\leq r\}$, the neighbors of $(1,1)$. Some pairs of these terminals are adjacent in $K_t\times K_r$; in that case we use the edge between them for the immersion. It remains to define a path between each pair of vertices in $N\left[(1,1)\right]$ that share a first coordinate or a second coordinate.

Define a graph $S$ with vertex set $N\left[(1,1)\right]$ where two vertices are adjacent if and only if they are not adjacent in $K_t\times K_r$. Note that a vertex $(x,y)\in S$ is part of one clique of size $r-1$, namely the clique with vertex set $\{(x, j): 2\leq j\leq r \}$, and part of one clique of size $t-1$, namely the clique with vertex set $\{(i, y): 2\leq i\leq t \}$; $(x,y)$ has no other adjacencies beyond these two cliques. Hence the edges of $S$ can be partitioned into $t-1$ copies of $K_{r-1}$ (one corresponding to each $i\in\{2, 3, \ldots, t\}$ in the first slot, which we call the $i$th copy of $K_{r-1}$) and $r-1$ copies of $K_{t-1}$ (one corresponding to each $j\in\{2, 3, \ldots, r\}$ in the second slot, which we call the $j$th copy of $K_{t-1}$).
In particular, $S$ is isomorphic to $K_{t-1}\Box K_{r-1}$.

For each edge in $S$, we must define a path in $K_t\times K_r$ between its endpoints. To do this we shall rely on edge-colorings of cliques, and associate colored edges in $S$ with particular paths to use in $K_t\times K_r$. A complete graph $K_{n}$ has maximum degree $n-1$ and so is $n$-edge-colorable by Vizing's Theorem. In the case that $n$ is odd and an $n$-edge-coloring using the colors $1, 2, \ldots n$ has been assigned to $K_n$, observe that each of these $n$ colors is missing at exactly one vertex. Remove the vertex that is missing an edge colored 1, and label the other vertices $2, 3, \ldots, n$ according to the color of its removed edge. We are now left with an $n$-edge-coloring of $K_{n-1}$ (an even clique) in which every vertex sees the color 1, and every other color is missing at exactly two vertices. In particular, this means that every vertex is missing exactly two of the colors $2, \ldots, n$, exactly one of which is its vertex label. We shall refer to this particular edge-coloring and vertex-labelling as our \emph{even clique assignment}; see Figure \ref{EvenCliqueAsst}. Given a copy of $K_{n-1}$ where $n$ is even (so $K_{n-1}$ is an odd clique), consider the $(n-1)$-edge-coloring of $K_{n-1}$ using the colors $2, 3, \ldots, n$. Each of these colors will be missing at exactly one vertex; consider each vertex to be labelled with its missing color. We shall refer to this particular edge-coloring and vertex-labelling as our \emph{odd clique assignment}; see Figure \ref{OddCliqueAsst}.

\begin{figure}
\centering
\includegraphics[height=5.5cm]{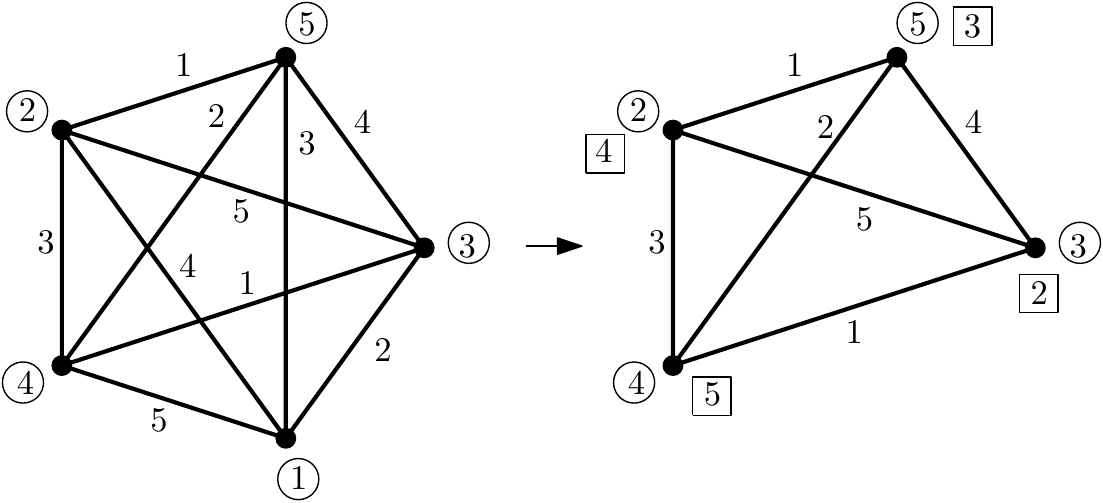}
\caption{An \emph{even clique assignment} for $K_4$. The circled colors are missing at the indicated vertices; the boxed colors are also missing but additionally serve as vertex-labels.}
\label{EvenCliqueAsst}
\end{figure}

\underline{Case 1:} $t, r$ are both even.

In this case, $K_{t-1}$ and $K_{r-1}$ are both odd cliques, and we use our odd-clique assignment on each of our $r$ copies of $K_{t-1}$ and each of our $t$ copies of $K_{r-1}$. We do this in such a way that vertex $(i,j)$ in $S$ is labelled $i$ in its copy of $K_{t-1}$ and labelled $j$ in its copy of $K_{r-1}$.

Suppose the color of the edge $(i,j) - (k,j)$ is $a$. Then we choose the path between these vertices to be
\[(i,j) -(a,1) -(k,j) .\]
Note that this is indeed a path in $K_t\times K_r$, as $a\neq i, k$ and $j\neq 1$. Since  $a\neq 1$, we are not using any of the (already used) edges incident to $(1, 1)$. Moreover, the edges in the $j$th-copy of $K_{t-1}$ labelled $a$ form a matching, so these paths use each edge incident to $(a,1)$ at most once.

Similarly, suppose the color of the edge $(i,j)-(i,k)$ is $b$. Then $b\neq 1, j, k$ and $i\neq 1$, and we choose the path from $(i, j)$ to $(i,k)$ to be
\[(i,j) -(1,b) -(i,k) .\]
The edges in the $i$th copy of $K_{r-1}$ labelled $b$ form a matching, so we use each edge incident to $(1,b)$ at most once. These edges are disjoint from the edges incident to $(a, 1)$ because $a$ and $b$ are not 1. This completes the description of our desired immersion in $K_t\times K_r$.

\begin{figure}
\centering
\includegraphics[height=5.5cm]{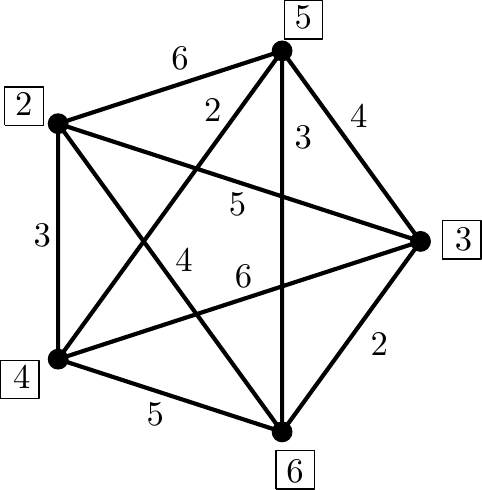}
\caption{An \emph{odd clique assignment} for $K_5$. The boxed colors are missing and additionally serve as vertex-labels.}
\label{OddCliqueAsst}
\end{figure}

\underline{Case 2:} exactly one of $t,r$ is odd.

Suppose, without loss of generality, that $t$ is even and $r$ is odd. In this case, $K_{t-1}$ is an odd clique while $K_{r-1}$ is an even clique. We use our odd-clique assignment on each of our $r$ copies of $K_{t-1}$ and our even-clique assignment on each of our $t$ copies of $K_{r-1}$. We do this in such a way that vertex $(i,j)$ in $H$ is labelled $i$ in its copy of $K_{t-1}$ and labelled $j$ in its copy of $K_{r-1}$.

For edges in copies of $K_{t-1}$, we do the path-assignment according to colors exactly as in Case 1.

Consider now an edge $(i,j)-(i,k)$ in the $i$th copy of $K_{r-1}$, and suppose it has color $b$. If $b\neq 1$, we define the path as before, namely
\[(i,j) -(1,b) -(i,k) .\]
In the case $b=1$ however, we must proceed differently as all edges incident to $(1,1)$ have already been used. In this case, we look more closely at this copy of $K_{r-1}$. The vertex $(i,j)$ is missing exactly two colors in this copy, namely $j$ and a second color $c\neq 1$. The vertex $(i,k)$ is missing $k$ and a second color $d\neq 1$. We choose the path between these vertices to be
$$(i,j)-(1, c) -(i,1)-(1,d) -(i,k).$$
We haven't used the edges $(i,j) - (1,c) $ or $ (1,d)-(i,k)$ in the first step (dealing with edges not colored 1 in the $K_{r-1}$), because $c$ is missing at $(i,j)$ and $d$ is missing at $(i,k)$. We haven't used the edges $(1,c)-(i,1)$ or $(i,1)-(1,d)$ in the first step because $j,k\neq 1$.  Moreover, these new paths do not overlap any of the edges used for our paths from the $K_{t-1}$'s (i.e. paths between vertices in one of the copies of $K_{t-1}$), because those edges were all of form $(x,y)- (a,1)$ where $x\neq 1$. This completes the description of our desired immersion in $K_t\times K_r$.

\underline{Case 3}: $t,r$ are both odd.

Let $t'=t-1$, so $t'$ is even. Define $S'$ to be the subgraph of $S$ obtained by deleting the vertices with $t$ in the first coordinate. Apply Case 2 to the pair $t', r$ to get paths corresponding to every edge in $S'$. Of these paths, we will change only the longest ones, that is, the paths of length 4. The paths of length 4 in Case 2 occur between vertices $(i,j)$ and $(i,k)$ when the color $b$ on the edge $(i,j)-(i,k)$ is 1. We will replace each such path with
$$(i,j)-(t,1)-(i,k).$$
Since $j,k\neq 1$ and since $t\neq i$, this is indeed a path. Since the paths we are replacing correspond to a matching (in fact a perfect matching in the $i$th copy of $K_{r-1}$) and since $t$ is a completely new value, these edges have not yet been used in the immersion.

It remains now to define paths between pairs of vertices in which at least one vertex has $t$ in the first coordinate. We will do this based on the edge-coloring of $S'$.

In particular, for a vertex $(t, j)$, $2\leq j \leq r$, we must define paths to it's neighbors in $S'$ and to the other vertices with first coordinate $t$. Let $2\leq i\leq t-1$.

For the first type of path, we use
$$(t,j)-(i,1)-(1,c)-(i,j),$$
where $c$ is the color missing at $(i,j)$ (in addition to $j$) in the $i$th copy of $K_{r-1}$. Note that the path $(i,1)-(1,c)-(i,j)$ was precisely one half of the length 4 path between $(i,j)$ and $(i,k)$ that we deleted. Hence these edges are indeed available and form a path (note the other half of this length four path will be used to join $(t, j)$ to $(i,k)$). The first edge of the path, $(t, j) - (i, 1)$ is an edge because $t\neq i$ and $j\neq 1$.

We must now define paths between each pair of vertices with $t$ in the first coordinate.

We applied Case 2 to $S'$, this means each copy of $K_{r-1}$ in $S'$ has the same fixed coloring of it's edges. Let $b$ be the color of the edge between $j$ and $k$ in $K_{r-1}$. If $b\neq 1$
then we define the path
\[(t,j) -(1,b) -(t,k) .\]
Note that, no edges of form $(t, x), (1,y)$ with $x,y\neq 1$ have been previously used.

If $b=1$, we must proceed differently, as all edges incident to $(1,1)$ have already been used. In this case, we look more closely at the edge-coloring of $K_{r-1}$. Each vertex, $j$, in $K_{r-1}$ is missing exactly two colors, namely $j$ and a second color $c\neq 1$. Another vertex $k$ is missing $k$ and a second color $d\neq 1$. We choose the path to be
$$(t,j)-(1, c) -(t,1)-(1,d) -(t,k).$$
The edge $(t,j)-(1,c)$ or $(1,d)-(t,k)$ have not previously been used because $c$ is missing at $j$ and $d$ is missing at $k$ in the copy of $K_{r-1}$. We previously used edges incident to $(t,1)$ in paths of form $(i,j)-(t,1)-(i,k)$, but there we know that $i\neq 1$, so we are not re-using any edges from those paths.
This completes the description of our desired immersion in $K_t\times K_r$.
\end{proof}

The following corollary follows directly from Theorem~\ref{KtTimesKr}.

\begin{corollary}\label{CompleteSubgraphTimes} Let $G$ and $H$ be graphs with $\im(G)=t$ and $\im(H)=r$, and suppose that $K_t$ is a subgraph of $G$ and $K_r$ is a subgraph of $H$. Then $\im(G\times H) \geq (t-1)(r-1)+1.$
\end{corollary}

\begin{proof} Define a $K_{(t-1)(r-1)+1}$-immersion in $G\times H$ using only the complete subgraphs and Theorem~\ref{KtTimesKr}. Therefore $\im(G\times H) \geq (t-1)(r-1)+1.$
\end{proof}

We now prove the conjecture for a general graph $G$ and a graph that contains $K_r$ as a subgraph. Here $r\geq 3$ since we use an $r\times r$ idempotent Latin Square in our proof of Case 3 and there is no $2\times 2$ idempotent Latin Square.

\begin{theorem} \label{GtimesKr} Let $G$ and $H$ be graphs with $\im(G)=t$ and $\im(H)=r$ where $r\geq 3$, and suppose $K_r$ is a subgraph of $H$. Then  $\im(G\times H) \geq (t-1)(r-1)+1$.
\end{theorem}

\begin{proof}
Note that, it suffices to prove the theorem for $H=K_r$. Let $v_1, v_2, \ldots, v_t$ be the terminals of a $K_t$-immersion in $G$, and let the vertices of $H=K_r$ be $1,2,\ldots, r$. The $(t-1)(r-1)+1$ terminals of our clique immersion will be $(v_1, 1)$ and $(v_2, k), (v_3, k), \ldots, (v_t, k)$ for each $k\in \{2, 3, \ldots, r\}$.

We use the same plan for routes between terminals as in the proof of Theorem~\ref{KtTimesKr}, where each vertex $v_i$ replaces $i$ for $1\leq i\leq t$. However, vertices that were adjacent in $K_t$ may now be connected by a path. In order to complete the immersion, we need to show how to replace each edge in any route used in the proof of Theorem~\ref{KtTimesKr} by a path.

Consider $(v_i,m)-(v_j,n)$ where $i<j$ and $m\neq n$. We want to describe a path in $G\times K_r$ from $(v_i,m)$ to $(v_j,n)$. If $v_i$ is adjacent to $v_j$ in $G$, then these two vertices are adjacent. Otherwise, let $P_{i,j}$ be the path in $G$ between $v_i$ and $v_j$ in a fixed $K_t$-immersion. Let $P_{i,j} = v_i - p_1 - p_2 - p_3 - \cdots - p_a- v_j$.

\underline{Case 1:} The number of pegs, $a$, is even.

Then we use the route
\[(v_i, m) - (p_1,n)-(p_2,m)-(p_3,n) - \cdots - (p_{a-1}, n)- (p_a, m)- (v_j, n)\]

These paths will be edge-disjoint because the paths $P_{i,j}$ in $G$ are edge-disjoint and we alternate between the $m$th copy and the $n$th copy of $K_r$.

\underline{Case 2:} The number of pegs, $a$, is odd and $K_r$ is an odd clique.

Consider an edge coloring of $K_r$ using $r$ colors in which the color $k$ is missing at vertex $k$. This is possible because $K_r$ is an odd clique. In this coloring, suppose the color on the edge $mn$ is $\ell$, this means $\ell\neq m, n$. Then we use the route
\[(v_i, m) - (p_1,\ell)-(p_2,m)-(p_3,\ell) - \cdots - (p_{a-1}, m)- (p_a, \ell)- (v_j, n)\]

These paths will be edge-disjoint because the paths $P_{i,j}$ in $G$ are edge-disjoint and the color $\ell$ is missing at vertices $m$ and $n$ in $K_r$.

\underline{Case 3:} The number of pegs, $a$, is odd and $K_r$ is an even clique.

Let $A$ be a $r\times r$ idempotent Latin square, that is, a Latin square in which $a_{hh}=h$. Consider a copy of $K_r$ in which every edge is replaced with a directed $2$-cycle. We use $A$ to color this digraph. color the directed edge $hk$ with $a_{hk}$.
 Suppose $\ell$ is the color on the directed edge $mn$. Since $A$ is an idempotent Latin square, $\ell\neq m, n$. Then we use the route
\[(v_i, m) - (p_1,\ell)-(p_2,m)-(p_3,\ell) - \cdots - (p_{a-1}, m)- (p_a, \ell)- (v_j, n)\]

Using a Latin square insures that each out-edge at a vertex is a different color and each in-edge at a vertex is a different color because the colors do not repeat in a row or column. This, combined with the fact that the $P_{i,j}$ are edge disjoint in $G$, means these paths will be edge-disjoint.
\end{proof}

The proof technique from Theorem~\ref{GtimesKr} does not work for $r=2$ because in this case we may not be able to choose $(v_1, 1), (v_2, 2), (v_3, 2), \ldots, (v_t, 2)$ as our terminals. For example, consider the graph in Figure~\ref{GcrossK2}. In using our proof technique to find an immersion of $K_4$ in $G\times K_2$ we would choose $(v_1, 1), (v_2, 2), (v_3, 2), (v_4, 2)$ as our terminals. Each terminal has degree 3 and therefore every edge incident to a terminal must be used on a path to connect to the other terminals. Every vertex in $G\times K_2$ has degree $2$ or $3$, this means every non-terminal vertex can be used as a peg at most once. Since $(v_3, 2)$ and $(v_4, 2)$ have two neighbors in common, namely $(v_2, 1)$ and $(p_1, 1)$,  at most one can be used on the path from $(v_3, 2)$ to $(v_4, 2)$, meaning the other neighbor must be used as a peg twice, once for  $(v_3, 2)$ and once for $(v_4, 2)$ contradicting it can only be used as a peg one time. Thus we are unable to complete the immersion. However, in Figure~\ref{GcrossK2}, we identify an immersion of $K_4$ in $G\times K_2$ using different terminals.
\begin{figure}
\centering
\includegraphics[width=.78\textwidth]{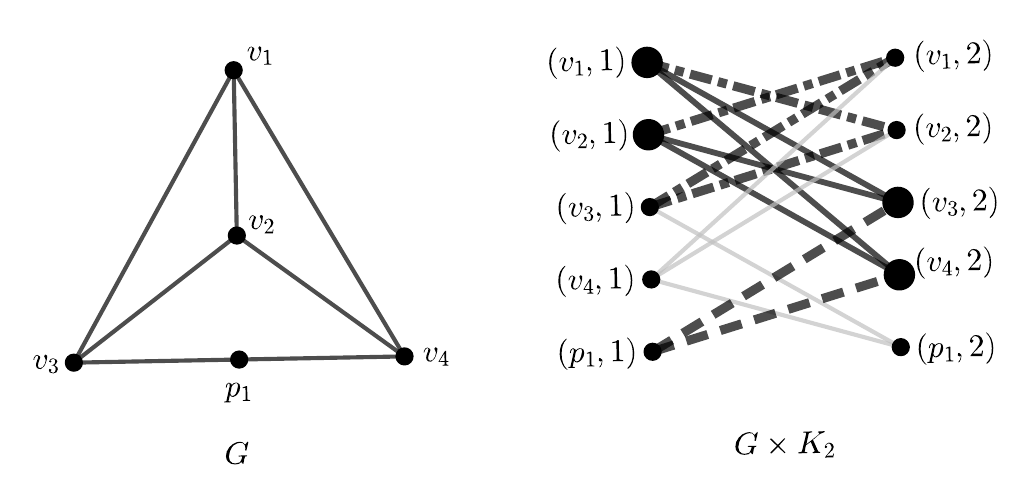}
\caption{An example of $G\times K_2$. $G$ has a $K_4$-immersion with terminals $v_1, v_2, v_3,$ and $v_4$, and we have indicated a $K_4$-immersion in $G \times K_2$. The terminals are the larger vertices and the gray lines are the edges unused by our $K_4$-immersion. The vertices $(v_3, 2)$ and $(v_4, 2)$ are joined by the path of length two through $(p_1, 1)$, and there is a copy of $K_{2,2}$ between $(v_1, 1), (v_2, 1)$ and $(v_3, 2), (v_4, 2)$. The vertices $(v_1, 1)$ and $(v_2, 1)$ are joined by the path of length four indicated by the dashed lines.
}
\label{GcrossK2}
\end{figure}

Another issue with the example provided in Figure 5 is that the immersion of $K_4$ in $G$ contains paths of different parity. We discuss this in the next section.

\subsection{Path parity}\label{PegParity}
In the proof of Theorem~\ref{KtTimesKr}, we found paths between two different kinds of pairs of vertices in $G\times H$: the first type is $(a,b)$ and $(c,d)$ where $a\neq c$ and $b\neq d$, and the second type is $(a,b)$ and $(a,d)$ or $(a,b)$ and $(c,b)$. In the first case, there were edges between the first coordinates in $G$, and between the second coordinates in $H$. In the second case, there were edges between one set of coordinates, and equality in the other coordinate. Any path between $(a,b)$ and $(a,d)$ requires a closed walk between the first coordinates and an open (i.e. not closed) walk in the second, and these walks must contain the same number of edges. If the open walk has an odd number of pegs, then the closed walk can alternate between $a$ and any neighbor of $a$, but if the open walk has an even number of pegs, then the closed walk must contain an odd number of edges, and hence contain an odd cycle. Any generalization of the theorem must therefore take into account the parity of the number of pegs between terminal vertices in the factors of a direct product.

In the next theorem we prove that if the parity of all paths in both immersions is the same, then our previous bound holds.

\begin{theorem}\label{DirectParity} Let $G$ and $H$ be graphs such that $\im(G)=t$ and $\im(H)=r$. If there is an immersion $I_1$ of $K_t$ in $G$ and an immersion $I_2$ of $K_r$ in $H$ such that every path between terminals in both of these immersions has the same parity, then $\im(G\times H) \geq (t-1)(r-1)+1$.
\end{theorem}

\begin{proof} Let $G$ and $H$ be graphs such that $\im(G)=t$ and $\im(H)=r$. Let immersion $I_1$ of $K_t$ in $G$ and  immersion $I_2$ of $K_r$ in $H$ be immersions such that the parity of every path in $I_1$ and $I_2$ between terminals is the same. Let the terminals of $I_1$ be labeled $1, 2, \ldots, t$ and the terminals of $I_2$ be labeled $1, 2, \ldots, r$. The $(t-1)(r-1)+1$ terminals of our clique immersion will be $(1, 1)$ and $(j, k)$ for each $j \in \{ 2, \ldots, t\}$ and each $k \in \{2, \ldots, r\}$.

\underline{Case 1:} Suppose every path in $I_1$ and $I_2$ has an even number of pegs. We will use the proof method of Theorem~\ref{KtTimesKr}. In Theorem~\ref{KtTimesKr} we constructed a $K_{(t-1)(r-1)+1}$-immersion, $I$, in $K_t\times K_r$ with terminals $(1,1)$ and all of its neighbors. We are using the same terminals now in $G\times H$, but will replace each edge of $I$ by a path in $G\times H$.

Let $(a, b)-(c, d)$ be an edge in $I$, thus $a \neq c$ and $b \neq d$. Let the path from $a$ to $c$ in $I_1$ be
$$P_{a, c} =  a - p_1 - p_2 - \ldots - p_k  - c$$
and let the path from $b$ to $d$ in $I_2$ be
$$P_{b, d} =  b - q_1 - q_2 - \ldots - q_l - d,$$
 where $k$ and $l$ are even and, without loss of generality, $k\leq l$. We choose the path from $(a, b)$ to $(c, d)$ in $G\times H$ to be
$$(a, b) - (p_1, q_1) - (p_2, q_2) - \ldots - (p_k, q_k) - (p_{k-1}, q_{k+1}) - (p_k, q_{k+2}) - \ldots - (p_k, q_l) - (c, d).$$

We must now confirm that these newly defined paths in $G\times H$ are edge-disjoint. Note that the only times we will use $P_{a,c}$ and $P_{b,d}$ together is when connecting $(a, b)$ to $(c, d)$ or when connecting $(a, d)$ to $(c, b)$. Using the above, our path from $(a, d)$ to $(c, b)$ will be
$(a, d) - (p_1, q_l) - (p_2, q_{l-1}) \ldots - (p_k, q_{l-k+1}) - (p_{k-1}, q_{l-k}) - (p_{k-1}, q_{l-k-1}) - \ldots - (p_k, q_1) - (c, b).$
This path is edge-disjoint from the path from $(a, b)$ to $(c, d)$ because all of the vertices are different (the sum of the subscripts of each vertex on the path from $(a, b)$ to $(c, d)$ is even, while the sum of the subscripts of each vertex on the path from $(a, d)$ to $(c, b)$ is odd). In the case where one of the paths, $P_{ac}$ or $P_{bd}$,  is an edge we give the first vertex a subscript of $0$ and the second vertex a subscript of $1$ and the above parity argument applies. Therefore we have defined a $K_{(t-1)(r-1)+1}$-immersion in $G\times H$.

\underline{Case 2:} Suppose every path in $I_1$ and $I_2$ has an odd number of pegs. We must define paths $(a, b) - (c, d)$ and $(a, b) - (1,1)$, where $a$ and $c$ are terminals in $I_1$ and $b$ and $d$ are terminals in $I_2$. Let the path from $a$ to $c$ in $I_1$ be
$$P_{a,c} =  a - p_1 - p_2 - \ldots - p_k  - c,$$
the path from $b$ to $d$ in $I_2$ be
    $$P_{b,d} =  b - q_1 - q_2 - \ldots - q_l - d,$$
the path from $1$ to $a$ in $I_1$ be
    $$P_{1,a} =  1 - w_1 - w_2 -\ldots - w_m - a,$$
and the path from $1$ to $b$ in $I_2$ be
 $$P_{1,b} =  1 - z_1 - z_2 - \ldots - z_n - b$$
where $k, l, m,$ and $n$ are odd and, without loss of generality, $k + 2j = l$ and $m + 2h = n$ for some whole numbers $j$ and $h$.
If $a \neq c$, $b\neq d$, then for our path from $(a,b)$ to $(c,d)$ we use the path

$(a, b) - (p_1, q_1) - (p_2, q_2) - \ldots - (p_k, q_k) - (p_{k-1}, q_{k+1}) - (p_k, q_{k+2}) - \ldots - (p_k, q_l) - (c, d).$

\noindent For our path from $(1, 1)$ to $(a, b)$ we use

$(1, 1) - (w_1, z_1) - (w_2, z_2) - \ldots - (w_m, z_m) - (w_{m-1}, z_{m+1}) - (w_m, z_{m+2}) - \ldots - (w_m, z_n) - (a, b).$

If $a = c$ then for our path from $(a, b)$ to $(a, d)$ we use
$$(a, b) - (w_m, q_1) - (a, q_2) - (w_m, q_3) - \ldots - (w_m, q_l) - (a, d).$$
Similarly, if $b = d$ then for our path from $(a,b)$ to $(c, d)$ we use
$$(a, b) - (p_1, z_n) - (p_2, b) - (p_3, z_n) - \ldots - (p_k, z_n) - (c, b).$$

We must confirm that none of the defined paths share any edges. Suppose for a contradiction that the edge $(x_1, x_2) - (y_1, y_2)$ is used on two different paths in our immersion where $x_1y_1$ is on $P_{e,f}$ in $I_1$ and $x_2y_2$ is on $P_{g,h}$ in $I_2$, where $e$ and $f$ are terminals in $I_1$ and $g$ and $h$ are terminals in $I_2$. This means that we used the paths $P_{e,f}$ and $P_{g,h}$ in two instances, i.e., when making the path from $(e, g)$ to $(f, h)$ and when making the path from $(e, h)$ to $(f, g)$. For the purpose of our argument, we label the paths $P_{e,f}$ and $P_{g,h}$ as follows
$$P_{e,f} = e - u_1 - u_2 - \ldots - x_1 - y_1 - \ldots - u_m - f$$ and
$$P_{g,h} = g - v_1 - v_2 - \ldots - x_2 - y_2 - \ldots - v_n - h$$
where $m$ and $n$ are odd and $m \leq n$.

If $e, f, g, h \neq 1$, then $e \neq f$ and $g \neq h$.
Then the path from $(e, g)$ to $(f, h)$ will be

$(e, g) - (u_1, v_1) - (u_2, v_2) - \ldots - (x_1, x_2) - (y_1, y_2) - \ldots - (u_m, v_m) - (u_{m-1}, v_{m+1}) - (u_m, v_{m+2}) - \ldots - (u_m, v_n) - (f, h)$

\noindent and the path from $(e, h)$ to $(f, g)$ will be

$(e, h) - (u_1, v_n) - (u_2, v_{n-1}) - \ldots - (x_1, y_2) - (y_1, x_2) - \ldots - (u_m, v_{n-m+1}) - (u_{m-1}, v_{n-m}) - \ldots (u_m, v_1) - (f, g)$.

\noindent As we can see, since the direction in which we traverse the path $P_{g,h}$ is different in each case, the edge $(x_1, x_2) - (y_1, y_2)$ is not actually repeated.

If $e=1$ then the paths we are considering are $(1, 1)$ to either $(f, g)$ or $(f, h)$ and $(f, g)$ to $(f, h)$. In each case the path $P_{e,f}$ is not used, so the edge $(x_1, x_2) - (y_1, y_2)$ will not be used.

We have shown that the paths we defined are edge disjoint and thus have defined an immersion of $K_{(t-1)(r-1)+1}$ in $G \times H$.
 \end{proof}

If we know that one of the factors of $G\times H$ has an immersion in which every path has an even number of pegs, we can generalize the result of Theorem~\ref{DirectParity} as follows. To do this we again use the blueprint provided by the proof of Theorem~\ref{GtimesKr}.

\begin{theorem}\label{GTimesParity} Let $G$ and $H$ be graphs such that $\im(G)=t$ and $\im(H)=r$ where $r\geq 3$. If there is an immersion $I_2$ of $K_r$ in $H$ such that every path in the immersion has an even number of pegs, then $\im(G\times H) \geq (t-1)(r-1)+1$
\end{theorem}

\begin{proof} Let $G$ and $H$ be graphs such that $\im(G)=t$ and $\im(H)=r$ and let $I_2$ be an immersion  of $K_r$ in $H$ such that every path in the immersion has an even number of pegs. We will use the proof method of Theorem~\ref{GtimesKr}. In Theorem~\ref{GtimesKr} we constructed a $K_{(t-1)(r-1)+1}$-immersion, $I$ in $G\times K_r$ with terminals $(v_1, 1)$ and $(v_2, k), (v_3, k), \ldots, (v_t, k)$ for each $k\in \{2, 3, \ldots, r\}$, where $v_1, v_2, \ldots, v_t$ are terminals in a $K_t$-immersion in $G$. We use the same terminals now in $G\times H$, but will replace each edge of $I$ by a path in $G\times H$.

Let $(a, b)-(c, d)$ be an edge in $I$, thus $a \neq c$ and $b \neq d$ and $ac$ is an edge in $G$.
Let the path from $b$ to $d$ in $I_2$ be
$$P_{b,d} =  b - q_1 - q_2 - \ldots - q_l - d,$$
 where $l$ is even  We choose the path from $(a, b)$ to $(c, d)$ in $G\times H$ to be
$$(a, b) - (c, q_1) - (a, q_2) - \ldots  - (a, q_l) - (c, d).$$

We must now confirm that these newly defined paths in $G\times H$ are edge-disjoint. Note that, the only times we will alternate $a$ and $c$ in the first coordinate and $P_{b,d}$ together is when connecting $(a, b)$ to $(c, d)$ or when connecting $(a, d)$ to $(c, b)$. Using the above, our path from $(a, d)$ to $(c, b)$ will be
$(a, d) - (c, q_l) - (a, q_{l-1}) - \ldots - (a, q_1) - (c, b)$. This path is edge-disjoint from the path from $(a, b)$ to $(c, d)$ because in this path $c$ is paired with $q_i$ where $i$ is even, while in the path from $(a, b)$ to $(c, d)$, $c$ is paired with $q_j$ where $j$ is odd.
\end{proof}

\subsection{Examples}\label{DirectEx} As a consequences to our theorems we now give the immersion numbers for the direct products of several specific families.

\begin{theorem} $\im(C_m\times K_r) = 2r-1$\end{theorem}

\begin{proof} When $r=2$ and $m$ is even, $C_m\times K_r$ is two disjoint copies of $C_m$ and thus $\im(C_m\times K_2) = 3$. When $r=2$ and $m$ is odd, $C_m\times K_r = C_{2m}$ and thus $\im(C_m\times K_2) = 3$. When $r\geq 3$, Theorem~\ref{GtimesKr} implies $\im(C_m\times K_r) \geq 2(r-1)+1=2r -1$ and Proposition~\ref{MaxProducts} implies $\im(C_m\times K_r) \leq 2(r-1)+1=2r -1$. Therefore $\im(C_m\times K_r) = 2r-1$.\end{proof}

\begin{theorem}\label{CmTimesCn} $\im(C_m\times C_n) = 5$.
\end{theorem}

\begin{proof}We have two cases: (1) $n$ is odd, and (2) $m$ and $n$ are even. Since $\Delta(C_m\times C_n) = 4$, Proposition~\ref{MaxProducts} implies $\im(C_m\times C_n) \leq 5$. Therefore in each case we need only show $\im(C_m \times C_n)\geq5$. For both cases we will label the vertices of $C_m$ with  $1, 2, \ldots, m$ and the vertices of $C_n$ with $1, 2, \ldots, n$ clockwise around the cycles.

\underline{Case 1:} Let $n$ be odd. By Thereom~\ref{GTimesParity}, if we can find an immersion $I_2$ of $K_3$ in $C_n$ in which the parity of all the paths are even, then $\im{(C_m \times C_n)} \geq 5$. We take vertices $1, 2,$ and $3$ as the terminals of our $K_3$-immersion in $C_n$. Since the cycle is odd, the number of pegs on each path in the immersion is even.

\underline{Case 2:} Let $m$ and $n$ be even. We divide this into three sub-cases, one where $m$ and $n$ are both greater than or equal to 6, one where $m=n=4$, and one where $m\geq 6$ and $n=4$.
\begin{enumerate}
  \item[(i)] Let $m$ and $n$ be greater than or equal to $6$. We can find immersions $I_1$ of $K_3$ in $C_m$ and an immersion $I_2$ of $K_3$ in $C_n$ in which the parity of all the paths is the same by taking vertices $1, 3,$ and $5$ as the terminals of our $K_3$-immersion in each cycle. Since the cycles are even, the number of pegs on each path in the immersion is odd. Therefore, by Theorem~\ref{DirectParity}, $\im{(C_m \times C_n)} \geq 5$.

  \begin{figure}
\centering
\includegraphics[width=\textwidth]{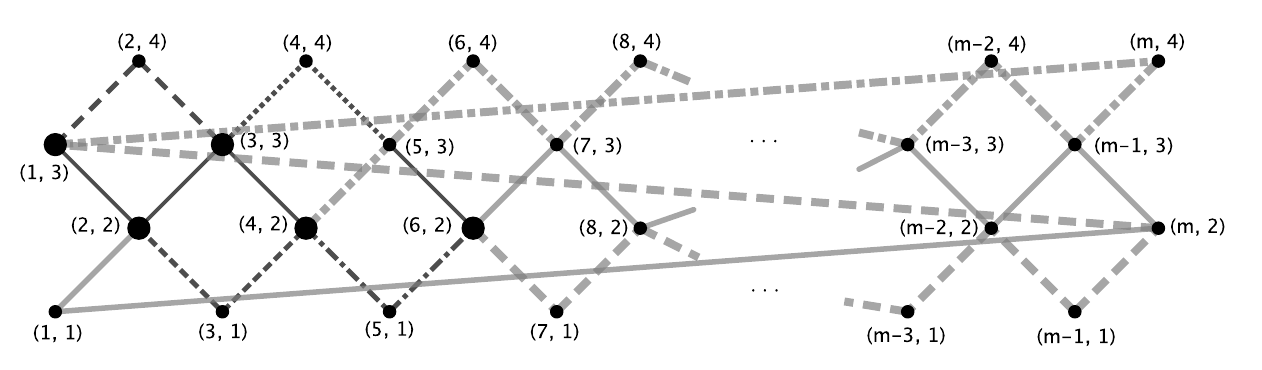}
\caption{A $K_5$-immersion in $C_{m}\times P_4$ for $m$ even and greater than or equal to $6$. Only one of the connected components of $C_{m}\times P_4$ is shown. Terminals are $(2,2), (4, 2), (6, 2), (1, 3),$ and $(3,3)$. Edges between terminals are shown as solid black lines. Paths are shown as various dashed lines, some gray and some black to indicate the different paths between terminals.
}
\label{C2jcrossP4}
\end{figure}

  \item[(ii)] Let $m = n = 4$. Since $C_4$ is bipartite, $C_4\times C_4$ has two isomorphic connected components. Vertices whose coordinates sum is even are in one component and those whose sum is odd are in another component. We will only describe the immersion for the even-sum component, where we use $(2, 2), (4, 2), (1, 3), (3, 3)$ and $(1, 1)$ as our terminals. We then use the following edges and paths to complete the immersion.

  $(2,2) - (3, 1) - (4, 2)$

  $(2,2) - (1,3)$

  $(2, 2) - (3,3)$

  $(2, 2) - (1, 1)$

  $(4,2) - (1, 3)$

  $(4,2) - (3, 3)$

  $(4,2) - (1, 1)$

  $(1, 3) - (2, 4) - (3,3)$

  $(1, 3) - (4, 4) - (3,1) - (2, 4) - (1,1)$

  $(3, 3) - (4, 4) - (1, 1)$

  \item[(iii)] Let $m\geq 6$ and $n=4$. Since $C_6$ and $C_4$ are bipartite, $C_6\times C_4$ has two isomorphic connected components. We will only describe the immersion for one component, where we use $(2, 2), (4, 2), (6, 2), (1, 3),$ and $(3, 3)$ as the terminals of our $K_5$-immersion. We then use the following edges and paths to complete the immersion.

  $(2,2) - (3, 1) - (4, 2)$

  $(2, 2) - (1, 1) - (m, 2) - (m-1, 3) - (m-2, 2) - (m-3, 3) - \ldots - (6, 2)$

  $(2,2) - (1,3)$

  $(2, 2) - (3,3)$

  $(4, 2) - (5, 1) - (6, 2)$

  $(4, 2) - (5, 3) - (6, 4) - (7, 3) - (8, 4) - \ldots (m-1, 3) - (m, 4) - (1, 3)$

  $(4,2) - (3, 3)$

  $(6, 2) - (5, 3) - (4, 4) - (3,3)$

  $(6,2) - (7,1) - (8, 2) - (9, 1) - \ldots - (m -1, 1) - (m, 2) - (1, 3)$

  $(1, 3) - (2, 4) - (3,3)$

  \noindent See Figure~\ref{C2jcrossP4} for an illustration of this immersion.
\end{enumerate}
\end{proof}

As an example where the bound of Conjecture~\ref{direct} is not tight we prove Theorem~\ref{CmTimesPath}.

\begin{theorem}\label{CmTimesPath} For $m\geq 5$, $\im(C_m\times P_4)=5$.\end{theorem}

\begin{proof}
\underline{Case 1:} Let $m$ be even. In Case 2(iii) of Theorem~\ref{CmTimesCn} we did not use the full cycle in the second coordinate, so in fact this case proves that $\im(C_m\times P_4)=5$ for $m$ even and greater than or equal to $6$ (also illustrated in Figure~\ref{C2jcrossP4}).

\underline{Case 2:} Let $m$ be odd. Label the vertices of the cycle $1, 2, \ldots, m$ clockwise around the cycle and label consecutive vertices on the path $1, 2, 3, 4$. We choose  $(1, 2), (1, 3), (2, 2), (2, 3),$ and $(3, 2)$ as the terminals of our $K_5$-immersion. We then use the following edges and paths to complete the immersion.

$(1, 2) - (m, 3) - (m-1, 4) - (m-2, 3) - (m-3, 4) - \cdots - (1, 3)$

$(1, 2) - (m, 1) - (m-1, 2) - (m-2, 1) - (m-3, 2) - \cdots - (3, 1) - (2, 2)$

$(1, 2) - (2, 3)$

$(1, 2) - (2, 1) - (3, 2)$

$(1, 3) - (2, 2)$

$(1, 3) - (m, 4) - (m-1, 3) - (m-2, 4) - (m-3, 3) - \cdots - (2, 3)$

$(1, 3) - (m, 2) - (m-1, 3) - (m-2, 2) - (m-3, 3) - \cdots - (3, 2)$

$(2, 2) - (3, 3) - (4, 2) - (5, 3) - \cdots - (m, 3) - (1, 4) - (2, 3)$

$(2, 2) - (1, 1) - (m, 2) - (m-1, 1) - (m-2, 2) - \cdots - (3, 2)$

$(2, 3) - (3, 2)$

Therefore $\im(C_m\times P_4)=5$ for $m\geq 5$.
\end{proof}

\subsection{Limitations of proof techniques}\label{limitations} Using our proof techniques we cannot generalize the result of Theorem~\ref{GTimesParity} to an immersion in $H$ in which all paths have an odd number of pegs. We use the example in Figure~\ref{BipartiteEx} to illustrate this. Graphs $G$ and $H$ each have an immersion of $K_4$ and every path in $H$'s $K_4$-immersion has an odd number of pegs. Conjecture~\ref{Direct} predicts $G\times H$ will have an immersion of $K_{10}$. Given the degree of each vertex in $G\times H$ we see that the potential terminals for the $K_{10}$-immersion are
$$(1, 1), (1, 2), (1, 3), (1, 4), (2, 1), (2, 2), (2, 3), (2, 4),$$ $$(3, 1), (3, 2), (3, 3), (3, 4), (4, 1), (4, 2), (4, 3), (4, 4).$$
Since $G$ and $H$ are both bipartite graphs the product $G\times H$ has two connected components. The potential terminals are separated in different components: $(1, 1)$, $(1, 2)$, $(1, 3)$, $(1,4)$ are in one component, and all of the other potential terminals are in the other component. This means an immersion of $K_{10}$ would have to be in the component where $1$ is not in the first coordinate. Our proof techniques would have us choose one vertex and all of its ``neighbors," in the sense of $K_4\times K_4$, but this is not possible because none of the potential terminals have all of their ``neighbors" in the same component. Thus, when the paths in the immersions have different parities of paths we cannot use our proof techniques. However, the example in Figure~\ref{BipartiteEx} is not a counterexample to Conjecture~\ref{Direct} as we can find an immersion of $K_{10}$ using vertices $$(2, 1), (2, 2), (2, 3), (2, 4), (3, 1), (3, 3), (3, 4),  (4, 1), (4, 3), (4, 4)$$ as our terminals.

\begin{figure}
\centering
\includegraphics[width=.6\textwidth]{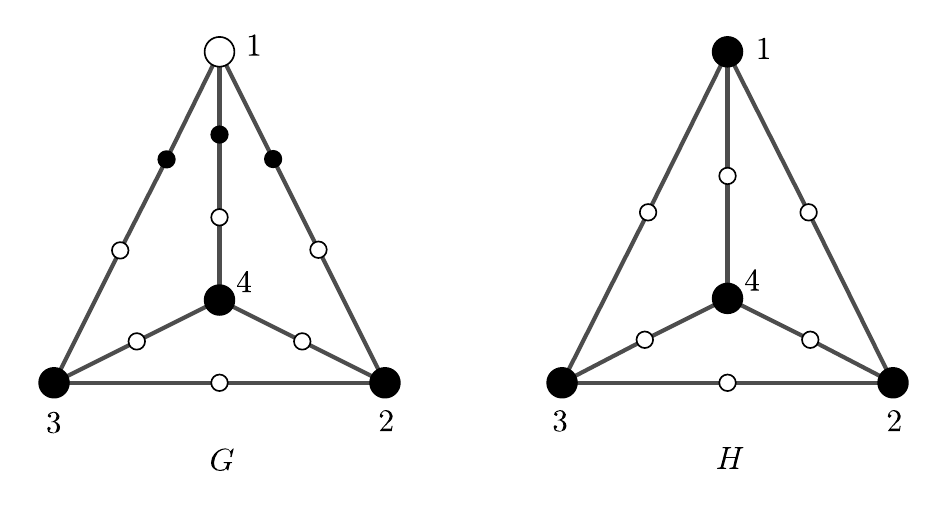}
\caption{Graphs $G$ and $H$ are bipartite graphs (with vertices colored black and while to show bipartitions), hence the product $G \times H$ has two connected components. Each of $G$ and $H$ has a $K_4$-immersion with terminals labeled $1, 2, 3, 4$ in the figure.
}
\label{BipartiteEx}
\end{figure}

\section{Final remarks}\label{final}

The last remaining product is the strong product. Recall that the edge set of the strong product is $E(G\Box H)\cup E(G\times H)$. Therefore, $K_t\boxtimes K_r=K_{tr}$ and $\im(K_t\boxtimes K_r)=tr$. This leads us to make the following conjecture about the immersion number of the strong product of two graphs.

\begin{conjecture}\label{boxtimes} Let $G$ and $H$ be graphs with $\im(G)=t$ and $\im(H)=r$ then $\im(G\boxtimes H)\geq tr$.\end{conjecture}

We believe Conjecture~\ref{boxtimes} will be resolved once the remaining cases for Conjecture~\ref{direct} are resolved.

In this paper we have completely resolved Question~\ref{ques} for the lexicographic and Cartesian products. We have provided much evidence that the answer will be yes for the direct product as well, but will need different proof techniques to fully resolve the direct product. For each product we were able to find examples where we could do better than the bound of $\im(K_t*K_r)$. For this reason, future work should include exploring the following question.

\begin{quest}\label{Nextques} Let $G$ and $H$ be graphs with $im(G)=t$ and $im(H)=r$. For each of the four standard graph products $G*H$, how large is $\im(G*H)$?
\end{quest}

\end{document}